\documentclass[11pt]{amsart}

\usepackage{graphicx}
\usepackage{mathptmx}

\usepackage{amscd}
\usepackage{amsxtra}
\usepackage{a4wide}
\usepackage{latexsym}
\usepackage{amssymb}
\usepackage{amsfonts}
\usepackage{amsmath}
\usepackage{amsrefs}
\usepackage{mathrsfs}
\usepackage{upref}
\usepackage{txfonts}
\usepackage{color}
\usepackage{tcolorbox}
\usepackage[utf8]{inputenc}

\usepackage[bookmarksnumbered, colorlinks, plainpages]{hyperref}
\hypersetup{colorlinks=true,linkcolor=red, anchorcolor=green,
	citecolor=cyan, urlcolor=red, filecolor=magenta, pdftoolbar=true}

\usepackage[left=2cm, right=2cm, top=2cm, bottom=2cm]{geometry}

\allowdisplaybreaks

\theoremstyle{plain}
\newtheorem{thm}{Theorem}[section]
\newtheorem{lem}[thm]{Lemma}

\newtheorem{prop}[thm]{Proposition}
\theoremstyle{definition}
\newtheorem{rem}[thm]{Remark}

\newtheorem{defi}[thm]{Definition}

\numberwithin{thm}{section}
\numberwithin{equation}{section}

\newcommand{\zstroke}{%
	\text{\ooalign{\hidewidth -\kern-.4em-\hidewidth\cr$\mathcal Z$\cr}}%
}

\def\supp{\operatorname{supp}}

\def\esup{\operatornamewithlimits{ess\,sup}}

\def\R{\mathbb R}
\def\Z{\mathbb Z}
\def\z{\mathcal Z}

\def\mp{{\mathfrak M}}

\def\I{(0,\infty)}

\def\W{{\mathcal W}}

\begin{document}

\title{Another approach to weighted inequalities \\ for a superposition of Copson and Hardy operators}

\author[R.Ch. Mustafayev]{RZA MUSTAFAYEV}
\address{RZA MUSTAFAYEV, Department of Mathematics, Kamil \"{O}zda\u{g} Faculty of Science, Karamano\u{g}lu Mehmetbey University, 70200, Karaman, Turkey}
\email{rzamustafayev@gmail.com}

\author[M. Yilmaz]{MERVE YILMAZ}
\address{MERVE YILMAZ, Department of Mathematics, Kamil \"{O}zda\u{g} Faculty of Science, Karamano\u{g}lu Mehmetbey University, 70200, Karaman, Turkey}
\email{mervegorgulu@kmu.edu.tr}



\begin{abstract}
In this paper, we present a solution to the inequality
$$
\bigg( \int_0^{\infty} \bigg( \int_x^{\infty} \bigg( \int_0^t h \bigg)^q w(t)\,dt
\bigg)^{r / q} u(x)\,ds \bigg)^{1/r}\leq C \, \bigg( \int_0^{\infty} h^p v \bigg)^{1 / p}, \quad h \in {\mathfrak M}^+(0,\infty),
$$
using a combination of reduction techniques and discretization. Here $1 \le p < \infty$, $0 < q ,\, r < \infty$ and $u,\,v,\,w$ are weight functions on $(0,\infty)$.
\end{abstract}

\maketitle

\section{Introduction}\label{in}

Throughout this paper by ${\mathfrak M}^+ (0,\infty)$ we denote the set of all non-negative measurable functions on $(0,\infty)$.
A weight is a function $v \in {\mathfrak M}^+ (0,\infty)$ such that
$$
0 < \int_0^x v(t)\,dt < \infty \quad \mbox{for all} \quad x \in (0,\infty).
$$
The family of all weight functions (also called just weights) on $(0,\infty)$ is given by ${\mathcal W}\I$. In the following, assume that $u,\,v,\,w \in {\mathcal W}\I$.

The investigation of weighted iterated Hardy-type inequalities started with the study of the inequality
\begin{equation}\label{mainn0}
\bigg( \int_0^{\infty} \bigg(\int_0^x \bigg( \int_t^{\infty} h \bigg)^q w(t)\,dt \bigg)^{\frac{r}{q}} u(x)\,dx\bigg)^{\frac{1}{r}} \le C \bigg(\int_0^{\infty} h^p v \bigg)^{\frac{1}{p}}, \qquad h \in \mp^+(0,\infty). 
\end{equation}
Inequality \eqref{mainn0} have been considered in the case $q=1$ in \cite{gop2009} (see also \cite{g1}), where the result was
presented without proof, and in the case $p=1$ in \cite{gjop} and \cite{ss}, where the special
type of weight function $v$ was considered. Recall that the inequality has been completely characterized	in \cite{GMP1} and \cite{GMP2} in the case $0 < q < \infty$, $0 < r \leq \infty$, $1 \le p < \infty$ by using discretization and anti-discretization methods; but, the obtained results were restricted to non-degenerate weights. Another approach to get the characterization of inequality \eqref{mainn0} was presented in \cite{ProkhStep1}. However, this characterization involves auxiliary functions, which make conditions more complicated. 

As it was mentioned in \cite{GogMusIHI} the characterization of "dual" inequality
\begin{equation}\label{main}
\bigg( \int_0^{\infty} \bigg( \int_x^{\infty} \bigg( \int_0^t h \bigg)^q w(t)\,dt
\bigg)^{\frac{r}{q}} u(x)\,dx \bigg)^{\frac{1}{r}}\leq C \, \bigg( \int_0^{\infty} h^p v \bigg)^{\frac{1}{p}}, \qquad h \in {\mathfrak M}^+(0,\infty)
\end{equation}
can be easily obtained from the solutions of inequality \eqref{mainn0}, which was presented in \cite{GKPS}. 

Another pair of "dual" weighted iterated Hardy-type inequalities are
\begin{equation}\label{iterH1}
\bigg( \int_0^{\infty} \bigg(\int_x^{\infty} \bigg( \int_t^{\infty} h \bigg)^q w(t)\,dt \bigg)^{\frac{r}{q}} u(x)\,dx \bigg)^{\frac{1}{r}} \le C \bigg(\int_0^{\infty} h^p v \bigg)^{\frac{1}{p}}, \qquad h \in \mp^+ (0,\infty)
\end{equation}
and 
\begin{equation}\label{iterH3}
\bigg( \int_0^{\infty} \bigg(\int_0^x \bigg( \int_0^t h \bigg)^q w(t)\,dt \bigg)^{\frac{r}{q}} u(x)\,dx \bigg)^{\frac{1}{r}} \le C \bigg(\int_0^{\infty} h^p v \bigg)^{\frac{1}{p}}, \qquad h \in \mp^+(0,\infty).
\end{equation}

The characterization of all four inequalities can be reduced to the characterization of the weighted Hardy-type	inequalities on the cones of non-increasing or non-decreasing functions, when $1 < p< \infty$. This approach provides solution of iterated inequalities by so-called "flipped" conditions (see, \cite{GogMusIHI} and \cite{gog.mus.2017_2}). In the case when $p = 1$, \cite{GogMusIHI} contains solutions of inequalities \eqref{mainn0} - \eqref{iterH3} with weight functions $\int_0^x v$ and $\big( \int_x^{\infty} v\big)^{-1}$ on the right-hand side, as well.

Different approach to solve \eqref{iterH1} has been given in \cite{mus.2017} when $p=1$ using a combination of reduction techniques and discretization. The "classical" conditions ensuring the validity of \eqref{iterH1} was recently presented in \cite{krepick}. Inequalities \eqref{main} and \eqref{iterH3} were recently characterized by using discretization techniques in \cite{GMPTU} and \cite{GU}, respectively.

In the present paper we solve inequality \eqref{main} using a combination of reduction techniques and discretization, which essentially shortens the proof. Similar approach was applied to inequality \eqref{iterH1} in \cite{mus.2017} when $p=1$. Our approach consists of the following steps: We prove that
\begin{align*}
	\bigg( \int_0^{\infty} \bigg( \int_x^{\infty} \bigg( \int_0^t h \bigg)^q w(t)\,dt
	\bigg)^{\frac{r}{q}} u(x)\,dx \bigg)^{\frac{1}{r}} & \\
	& \hspace{-3cm}\thickapprox \bigg\| \bigg\{ 2^{\frac{k}{r}} \bigg(
	\int_{x_k}^{x_{k+1}} \bigg( \int_{x_k}^s h \bigg)^q
	w(s)\,ds\bigg)^{\frac{1}{q}} \bigg\} \bigg\|_{\ell^r(\z)}  \\
	& \hspace{-2.5cm} +  \bigg( \int_0^{\infty}  u(x) \bigg[ \sup_{s \in [x,\infty)} \bigg(
	\int_s^{\infty} w \bigg)^{\frac{1}{q}} \bigg( \int_0^s h\bigg) \bigg]^r \,dx\bigg)^{\frac{1}{r}} 
\end{align*}
with constants independent of $h \in {\mathfrak M}^+(0,\infty)$, where  $\{x_k\}_{k \in \z}$ is a covering sequence mentioned in Definition \ref{rem.disc.} (see Theorem \ref{thm.1.4}). Consequently, inequality \eqref{main} holds if and only if both inequalities
\begin{align}
\bigg\| \bigg\{ 2^{k / r} \bigg(
\int_{x_k}^{x_{k+1}} \bigg( \int_{x_k}^t h \bigg)^q
w(t)\,dt \bigg)^{\frac{1}{q}} \bigg\} \bigg\|_{\ell^r(\z)} & \le C \bigg\| \bigg\{ \bigg( \int_{x_n}^{x_{n+1}} h^p v \bigg)^{\frac{1}{p}} \bigg\} \bigg\|_{\ell^p (\z)}, \quad h \in {\mathfrak M}^+(0,\infty) \label{ineq.discr} \\
\intertext{and}
\bigg( \int_0^{\infty}  u(x) \bigg[ \sup_{s \in [x,\infty)} \bigg(
\int_s^{\infty} w \bigg)^{\frac{1}{q}} \bigg( \int_0^s h\bigg) \bigg]^r \,dx \bigg)^{\frac{1}{r}} & \leq C \, \bigg( \int_0^{\infty} h^p v \bigg)^{\frac{1}{p}}, \quad h \in {\mathfrak M}^+(0,\infty) \label{ineq.supr}
\end{align}
hold (see Theorem \ref{thm.1.5}). Recall that the solution of inequality \eqref{ineq.supr} is known (see Theorem \ref{thm44b+1}). 
Discrete characterization of inequality \eqref{ineq.discr} is presented in Section \ref{s.5} (see Lemma \ref{lem.1.6}). Continuous sufficient conditions for inequality \eqref{ineq.discr} are studied in Section \ref{s.6} (see Lemma \ref{lem.1.8}). In Section \ref{s.7}, we show that these conditions are necessary for inequality \eqref{main} (see Lemma \ref{lem.1.9}). The proof of the main statement is given in Section \ref{s.8}.

We pronounce that the characterizations of inequalities \eqref{mainn0} - \eqref{iterH3} are important because many inequalities
for classical operators  can be reduced to them. These inequalities play an important role in the theory of weighted Morrey-type
spaces and Ces\`{a}ro function spaces (see \cite{GKPS}, \cite{gmu_CMJ}, \cite{gmu_2017} and \cite{gogmusunv}). Note that using characterizations of weighted Hardy inequalities it is easy to obtain the characterization of the boundedness of bilinear Hardy-type inequalities (see, for instance, \cite{Krep} and \cite{BMU}).

For a given weight function $v$, $0 \le x < y \le \infty$ and $1 \le p < \infty$, set
\begin{equation*}
\sigma_p(x,y) : = \left\{\begin{array}{cl}
\left(\int_x^y v(t)^{1-p'}\,dt  \right)^{\frac{1}{p'}} & \qquad \mbox{when} \qquad 1 < p < \infty \\
\esup_{x < t < y} v(t)^{-1} & \qquad \mbox{when} \qquad p = 1.
\end{array}
\right.
\end{equation*}

Our main result, which coincides with \cite[Theorem A]{GMPTU}, reads as follows:
\begin{thm}\label{main1}
	Let $1 \le p < \infty$, $0 < q,\, r < \infty$, $u \in {\mathcal W}(0,\infty) \cap C(0,\infty)$ and  $v,\,w \in {\mathcal W}\I$.
	
	{\rm (a)} Let $p \le \min\{q,\,r\}$. Then inequality \eqref{main} holds if and only if $F_1 < \infty$ and $F_2 < \infty$, where
	\begin{equation*}
	F_1 := \sup_{x \in (0,\infty)} \bigg( \int_0^x u \bigg)^{\frac{1}{r}} \bigg( \int_x^{\infty} w \bigg)^{\frac{1}{q}} \sigma_p(0,x) < \infty,
    \end{equation*}
    and
    \begin{equation*}
    F_2 := \sup_{x \in (0,\infty)} \bigg( \int_x^{\infty} \bigg( \int_t^{\infty} w \bigg)^{\frac{r}{q}} u(t)\,dt \bigg)^{\frac{1}{r}} \sigma_p(0,x) < \infty.
    \end{equation*}
	Moreover, if $C$ is the best constant in \eqref{main}, then $C \approx F_1 + F_2$.
	
	{\rm (b)} Let $r < p \le q$. Then inequality \eqref{main} holds if and only if $F_3 < \infty$ and $F_4 < \infty$, where
	\begin{equation*}
	F_3 := \bigg( \int_0^{\infty} \bigg[\sup_{t \in [x,\infty)} \bigg( \int_t^{\infty} w \bigg)^{\frac{1}{q}} \sigma_p(0,t) \bigg]^{\frac{pr}{p-r}} \bigg( \int_0^x u \bigg)^{\frac{r}{p-r}} u(x) \,dx \bigg)^{\frac{p-r}{pr}} < \infty,
	\end{equation*}
	and
	\begin{equation*}
	F_4 := \bigg(\int_0^{\infty} \bigg(\int_{x}^{\infty} \bigg( \int_y^{\infty} w \bigg)^{\frac{r}{q}} u(y) \, dy\bigg)^{\frac{r}{p-r}} \bigg( \int_x^{\infty} w	\bigg)^{\frac{r}{q}} [\sigma_p(0,x)]^{\frac{pr}{p-r}} \, u(x)\,dx \bigg)^{\frac{pr}{p-r}} < \infty.
	\end{equation*}
	
	Moreover, if $C$ is the best constant in \eqref{main}, then $C \approx F_3 + F_4$.

	{\rm (c)} Let $q < p \le r$. Then inequality \eqref{main} holds if and only if $F_2 < \infty$ and $F_5 < \infty$, where
    \begin{equation*}
    F_5 : = \sup_{t \in (0,\infty)} \bigg( \int_0^t u \bigg)^{\frac{1}{r}} \bigg( \int_t^{\infty} \bigg( \int_x^{\infty} w \bigg)^{\frac{q}{p-q}} w(x) \big[ \sigma_p(0,x) \big]^{\frac{pq}{p-q}} \,dx\bigg)^{\frac{p-q}{pq}} < \infty.
    \end{equation*}

    Moreover, if $C$ is the best constant in \eqref{main}, then $C \approx F_2 + F_5$.	
    
    {\rm (d)} Let $\max\{q,\,r\} < p$. Then inequality \eqref{main} holds if and only if $F_4 < \infty$ and $F_6 < \infty$, where
    \begin{align*}
    F_6 & := \bigg( \int_0^{\infty} \bigg( \int_0^t u \bigg)^{\frac{r}{p-r}} u(t) \bigg( \int_t^{\infty} \bigg( \int_x^{\infty} w \bigg)^{\frac{q}{p-q}} w(x) \big[ \sigma_p(0,x) \big]^{\frac{pq}{p-q}} \,dx\bigg)^{\frac{r(p-q)}{q(p-r)}}\,dt \bigg)^{\frac{p-r}{pr}} < \infty.	
    \end{align*}
    
    Moreover, if $C$ is the best constant in \eqref{main}, then $C \approx F_4 + F_6$.	
\end{thm}


\section{Notations and Preliminaries}\label{pre}

Throughout the paper, we always denote by $C$ a positive
constant, which is independent of main parameters but it may vary
from line to line. However a constant with subscript or superscript such as $C_1$
does not change in different occurrences. By $a\lesssim b$,
($b\gtrsim a$) we mean that $a\leq \lambda b$, where $\lambda >0$ depends on
inessential parameters. If $a\lesssim b$ and $b\lesssim a$, we write
$a\approx b$ and say that $a$ and $b$ are  equivalent. Unless a
special remark is made, the differential element $dx$ is omitted
when the integrals under consideration are the Lebesgue integrals.
We put $0 \cdot \infty = 0$, $\infty / \infty =
0$ and $0/0 = 0$.

Let $\emptyset \neq \zstroke \subseteq \overline{\Z} : = \Z \cup \{
-\infty ,+\infty\}$, $0 < q \le +\infty$ and $\{w_k\} = \{w_k\}_{k \in	\zstroke}$ be a sequence of positive numbers. We denote by $\ell^q(\{w_k\},\zstroke)$ the following discrete analogue of a weighted
Lebesgue space: if $0 < q < + \infty$, then
\begin{align*}
&\ell^q(\{w_k\},\zstroke) = \bigg\{ \{a_k\}_{k \in \z}: \|\{a_k\}\|_{\ell^q(\{w_k\},\zstroke)}: = \bigg(\sum_{k \in \zstroke}|a_k w_k|^q\bigg)^{\frac{1}{q}}<+\infty \bigg\},\\
\intertext{and} &
\ell^\infty(\{w_k\},\zstroke) =\left\{ \{a_k\}_{k\in\z}: \|\{a_k\}\|_{\ell^\infty(\{w_k\},\zstroke)} : = \sup_{k\in\z}|a_kw_k|<+\infty \right\}.
\end{align*}
If $w_k=1$ for all $k \in \zstroke$, we write simply $\ell^q(\zstroke)$ instead of $\ell^q(\{w_k\},\zstroke)$. 

The following inequality is a straightforward consequence of the discrete H\"{o}lder inequality:
\begin{equation}\label{eq.1.1}
\| \{a_k b_k\} \|_{\ell^r (\zstroke)} \le \| \{a_k \} \|_{\ell^{\rho}(\zstroke)} \, \| \{b_k \} \|_{\ell^p (\zstroke)}.
\end{equation}

\begin{defi} \label{D:2.1}
	Let $N,M\in \overline{\Z}$, $N<M$. A positive sequence $\{\tau _k\} _{k=N}^M$ is called geometrically increasing if
	there is $\alpha \in (1, +\infty )$ such that
	$$
	\tau _k\geq \alpha \tau _{k-1} \quad \text{for all} \quad k\in
	\{ N+1, \dots ,M\}.
	$$
\end{defi}

Proofs of the following statement can be found in \cite{Le1} and \cite{Le2}.
\begin{lem}\label{L:1.2}
	Let $q\in (0,+\infty]$, $N, M \in \overline{\Z}$, $ N\le M$, $\zstroke = \{N,N+1,\ldots,M-1,M\}$ and let $\{ \tau_k\} _{k=N}^M$ be a	geometrically increasing sequence. Then
	
	\begin{equation*}
	\left\| \left\{\tau _k \sum_{m=k}^{M}a_m\right\}\right\|_{\ell^q(\zstroke)} \approx \| \{\tau _ka_k\}\| _{\ell^q(\zstroke)}
	\end{equation*}
	and
	
	\begin{equation*}
	\bigg\| \bigg\{\tau _k \sup _{k\leq m\leq M}a_m \bigg\}\bigg\|_{\ell^q(\zstroke)} \approx \|\{\tau _ka_k\}\| _{\ell^q(\zstroke)}
	\end{equation*}
	for all non-negative sequences $\{ a_k\} _{k=N}^M$.
\end{lem}

Given two (quasi-) Banach spaces $X$ and $Y$, we write $X
\hookrightarrow Y$ if $X \subset Y$ and if the natural embedding of
$X$ in $Y$ is continuous.

The discrete version of the classical Landau resonance theorems is given in the following proposition. Proofs can be found, for example, in
\cite{gp1}.
\begin{prop}\label{prop.2.1}{\rm(\cite[Proposition 4.1]{gp1})}
	Let $0 < p,\,r < +\infty$, $\emptyset \neq \zstroke \subseteq \overline{\Z}$ and let $\{v_k\}_{k \in \zstroke}$ and $\{w_k\}_{k \in \zstroke}$ be
	two sequences of nonnegative numbers. Assume that
	\begin{equation}\label{eq31-4651}
	\ell^p (\{v_k\},\zstroke) \hookrightarrow \ell^r (\{w_k\},\zstroke).
	\end{equation}
	Then
	\begin{equation*}\label{eq31-46519009}
	\big\|\big\{w_k v_k^{-1}\big\}\big\|_{\ell^\rho(\zstroke)} \le C,
	\end{equation*}
	where $1 / \rho : = ( 1 / r - 1 / p)_+$ \footnote{For any $a\in\R$ denote by $a_+ = a$ when $a>0$ and $a_+ = 0$ when $a \le 0$.} and $C$ stands for the norm of
	embedding \eqref{eq31-4651}.
\end{prop}

We will use the following well-known characterizations of weights for which the weighted Hardy-type inequality holds (see, for instance, \cite{ok}).
\begin{thm}\label{thm.Copson}
	Let $1 \le p < \infty$, $0 < q < \infty$ and $v,\,w \in {\mathcal W} (0,\infty)$. 
	
	{\rm (i)} Let $p \le q$. Then 
	$$
	\sup_{f \in \mp^+(0,\infty)} \frac{\bigg( \int_0^{\infty} \bigg( \int_0^x f \bigg)^q w(x)\,dx\bigg)^{\frac{1}{q}}}{\bigg( \int_0^{\infty} f^p v \bigg)^{\frac{1}{p}}} \approx \sup_{x \in (0,\infty)} \bigg( \int_x^{\infty} w \bigg)^{\frac{1}{q}} \sigma_p(0,x).
	$$
	
	{\rm (ii)} Let $q < p$. Then 
	$$
	\sup_{f \in \mp^+(0,\infty)} \frac{\bigg( \int_0^{\infty} \bigg( \int_0^x f \bigg)^q w(x)\,dx\bigg)^{\frac{1}{q}}}{\bigg( \int_0^{\infty} f^p v \bigg)^{\frac{1}{p}}} \approx \bigg( \int_0^{\infty} \bigg( \int_x^{\infty} w \bigg)^{\frac{q}{p-q}} \, w(x) \, \big[ \sigma_p(0,x) \big]^{\frac{pq}{p-q}} \,dx \bigg)^{\frac{p-q}{pq}}.
	$$
\end{thm}	

We next quote the following result concerning characterization of inequality involving supremum operator.
\begin{thm}\cite[Theorems 4.1 and 4.4]{gop}\label{thm44b}
	Let $1 \le p < \infty$, $ 0 < r < \infty $. Assume that $u \in \W (0,\infty) \cap C (0,\infty)$ and  $v,\,w \in \W (0,\infty)$.
	
	{\rm (a)} Let $ p \le r $. 	Then the inequality
	\begin{equation}\label{gogopick3}
	\bigg( \int_0^{\infty} \bigg[\sup_{y \in [x,\infty)} u(y) \int_0^y g(s)\, ds\bigg]^r w(x)\, dx \bigg)^{\frac{1}{r}} \le C \, \bigg( \int_0^{\infty} g(x)^p v(x) \, dx \bigg)^{\frac{1}{p}}
	\end{equation}
	holds for all $ g\in \mathfrak{M}^{+} (0,\infty)$ if and only if
	\begin{equation*}
	C_1 := \sup_{x \in (0,\infty)} \bigg( \int_0^x w \bigg)^{\frac{1}{r}} \bigg( \sup_{t \in [x,\infty)} u(t) \bigg) \sigma_p(0,x) < \infty
	\end{equation*}
	and
	\begin{equation*}
	C_2 := \sup_{x \in (0,\infty)} \bigg( \int_x^{\infty} \bigg( \sup_{\tau \in [t,\infty)} u(\tau) \bigg)^r w(t)\,dt \bigg)^{\frac{1}{r}} \sigma_p(0,x) < \infty.
	\end{equation*}
	Moreover, the least constant $C$ such that \eqref{gogopick3} holds for all $ g\in \mathfrak{M}^{+} $ satisfies $ C \approx C_1 + C_2$.
	
	{\rm (b)} Let $r < p$. Then inequality \eqref{gogopick3} holds for all $ g\in \mathfrak{M}^{+} (0,\infty)$ if and only if
	\begin{equation*}
	C_3 := \bigg(\int_0^{\infty} \bigg[\sup_{\tau \in [x,\infty)} \bigg[\sup_{y \in [\tau,\infty)}u(y) \bigg] \sigma_p(0,\tau) \bigg]^{\frac{pr}{p-r}} \bigg(\int_0^{x} w \bigg)^{\frac{p}{p-r}} w(x) \,dx \bigg)^{\frac{p-r}{pr}} < \infty
	\end{equation*}
	and
	\begin{equation*}
	C_4 := \bigg(\int_0^{\infty} \bigg(\int_{x}^{\infty} \bigg[\sup_{\tau \in [y,\infty)} u(\tau)\bigg]^r w(y) \, dy\bigg)^{\frac{r}{p-r}} \bigg[\sup_{\tau \in [x,\infty)} u(\tau)\bigg]^r [\sigma_p(0,x)]^{\frac{pr}{p-r}} \, w(x)\,dx \bigg)^{\frac{p-r}{pr}} < \infty.
	\end{equation*}
	Moreover, the least constant $C$ such that \eqref{gogopick3} holds for all $ g\in \mathfrak{M}^{+} $ satisfies $ C \approx C_3 + C_4$.
	
\end{thm}







\section{Equivalence and reduction theorems}\label{eq.red.thms}

In this section we prove the equivalence and reduction theorems.

\begin{defi}\label{rem.disc.}
	Assume that $u$ is a weight function on $(0,+\infty)$. If $\int_0^{\infty} u(t)\,dt = +\infty$, let $\{x_k\}_{-\infty}^{+\infty} \subset (0,\infty)$ be a strictly increasing sequence  such that $\int_0^{x_k} u(t)\,dt = 2^k$, $k \in \Z$. Denote $M:= +\infty$ and $\z = \Z$, when $\int_0^{\infty} u(t)\,dt = \infty$.  If $\int_0^{\infty} u(t)\,dt < +\infty$, define a strictly increasing sequence $\{x_k\}_{k = -\infty}^M$ such that $\int_0^{x_k} u(t)\,dt = 2^k$, $-\infty < k \le M$, where $M$ satisfies the inequality $2^M \le \int_0^{\infty} u(t)\,dt < 2^{M+1}$. Denote $x_{M+1} : = \infty$ and $\z : = \{k \in \Z:\, k \le M\}$, when $\int_0^{\infty} u(t)\,dt < \infty$.
   	Obviously, $\bigcup_{k \in \z} [x_k,x_{k+1}) = (0,\infty)$ in both cases. The sequence $\{x_k\}_{k \in \z}$ is called a covering sequence. 
\end{defi}

\begin{rem}
We shall use the following equivalences without mentioning anytime we need them.
	
Assume that $\{x_k\}_{k \in \z}$ is a covering sequence. Clearly, 
$$
\int_{x_{k-1}}^{x_k} u \approx 2^k, \quad k \in \z.
$$
Moreover,
$$
\int_{x_{k-1}}^{x_k} \bigg( \int_{x_{k-1}}^t u \bigg)^{\frac{r}{p-r}} u(t)\,dt \approx \int_{x_{k-1}}^{x_k} \bigg( \int_0^t u \bigg)^{\frac{r}{p-r}} u(t)\,dt  \approx 2^{k\frac{p}{p-r}}
$$
when $0 < r < p < \infty$.
\end{rem}

Our equivalency statement reads as follows:
\begin{thm}\label{thm.1.4}
	Let $0 < q,\, r < \infty$ and  $u,\,v,\,w \in {\mathcal W}\I$.	Assume that $\{x_k\}_{k \in \z}$ is a covering sequence.
	Then
	\begin{align*}
	\bigg( \int_0^{\infty} \bigg( \int_x^{\infty} \bigg( \int_0^t h \bigg)^q w(t)\,dt
	\bigg)^{\frac{r}{q}} u(x)\,dx \bigg)^{\frac{1}{r}} & \\
	& \hspace{-3cm} \thickapprox \, \bigg\| \bigg\{ 2^{\frac{k}{r}} \bigg(
	\int_{x_k}^{x_{k+1}} \bigg( \int_{x_k}^t h \bigg)^q
	w(t)\,dt \bigg)^{\frac{1}{q}} \bigg\} \bigg\|_{\ell^r(\z)} \\
	& \hspace{-2.5cm} +  \bigg( \int_0^{\infty}  u(x) \bigg[ \sup_{s \in [x,\infty)} \bigg(
	\int_s^{\infty} w \bigg)^{\frac{1}{q}} \bigg( \int_0^s h\bigg) \bigg]^r \,dx \bigg)^{\frac{1}{r}} 
	\end{align*}
	with constants independent of $h \in {\mathfrak M}^+(0,\infty)$.
\end{thm}

\begin{proof}
	
	Since
	\begin{align*}
	\bigg( \int_0^{\infty} \bigg( \int_x^{\infty} \bigg( \int_0^t h \bigg)^q w(t)\,dt
	\bigg)^{\frac{r}{q}} u(x)\,dx \bigg)^{\frac{1}{r}} & \notag \\
	& \hspace{-3cm} = \, \bigg\| \bigg\{ \bigg(\int_{x_k}^{x_{k+1}} \bigg( \int_x^{\infty} \bigg( \int_0^t h \bigg)^q w(t)\,dt \bigg)^{\frac{r}{q}} u(x)\,dx \bigg)^{\frac{1}{r}} \bigg\} \bigg\|_{\ell^r(\z)} \\
	& \hspace{-3cm} \le \, \bigg\| \bigg\{ \bigg(\int_{x_k}^{x_{k+1}} u \bigg)^{\frac{1}{r}} \bigg( \int_{x_k}^{\infty} \bigg( \int_0^t h \bigg)^q w(t)\,dt \bigg)^{\frac{1}{q}}  \bigg\} \bigg\|_{\ell^r(\z)} \\
	& \hspace{-3cm} = \, \bigg\| \bigg\{ 2^{\frac{k}{r}} \bigg( \int_{x_k}^{\infty}	\bigg( \int_0^t h \bigg)^q w(t) \,dt \bigg)^{\frac{1}{q}} \bigg\} \bigg\|_{\ell^r(\z)},
	\end{align*}
	and
	\begin{align*}
	\bigg( \int_0^{\infty} \bigg( \int_x^{\infty} \bigg( \int_0^t h \bigg)^q w(t)\,dt
	\bigg)^{\frac{r}{q}} u(x)\,dx \bigg)^{\frac{1}{r}} & \notag \\
	& \hspace{-3cm} \ge \, \bigg\| \bigg\{ \bigg(\int_{x_{k-1}}^{x_k} \bigg( \int_x^{\infty} \bigg( \int_0^t h \bigg)^q w(t)\,dt \bigg)^{\frac{r}{q}} u(x)\,dx \bigg)^{\frac{1}{r}} \bigg\} \bigg\|_{\ell^r(\z)} \\
	& \hspace{-3cm} \ge \, \bigg\| \bigg\{ \bigg(\int_{x_{k-1}}^{x_k} u \bigg)^{\frac{1}{r}} \bigg( \int_{x_k}^{\infty} \bigg( \int_0^t h \bigg)^q w(t)\,dt \bigg)^{\frac{1}{q}}  \bigg\} \bigg\|_{\ell^r(\z)} \\
	& \hspace{-3cm} \approx \, \bigg\| \bigg\{ 2^{\frac{k}{r}} \bigg( \int_{x_k}^{\infty}	\bigg( \int_0^t h \bigg)^q w(t) \,dt \bigg)^{\frac{1}{q}} \bigg\} \bigg\|_{\ell^r(\z)},
	\end{align*}
	then
	\begin{align*}
	\bigg( \int_0^{\infty} \bigg( \int_x^{\infty} \bigg( \int_0^t h \bigg)^q w(t)\,dt
	\bigg)^{\frac{r}{q}} u(x)\,dx \bigg)^{\frac{1}{r}} & \notag \\
	& \hspace{-3cm} \approx \, \bigg\| \bigg\{ 2^{\frac{k}{r}} \bigg( \int_{x_k}^{\infty} \bigg( \int_0^t h \bigg)^q w(t) \,dt \bigg)^{\frac{1}{q}} \bigg\} \bigg\|_{\ell^r(\z)}.
	\end{align*}
	
	By Lemma \ref{L:1.2}, we get that
	\begin{align}
	\bigg( \int_0^{\infty} \bigg( \int_x^{\infty} \bigg( \int_0^t h \bigg)^q w(t)\,dt
	\bigg)^{\frac{r}{q}} u(x)\,dx \bigg)^{\frac{1}{r}} & \notag \\	
	& \hspace{-3cm}  \approx \, \bigg\| \bigg\{ 2^{\frac{k}{r}} \bigg( \int_{x_k}^{x_{k+1}}	\bigg( \int_0^t h \bigg)^q w(t) \,dt \bigg)^{\frac{1}{q}} \bigg\} \bigg\|_{\ell^r(\z)} \notag \\
	& \hspace{-3cm}  \approx \, \bigg\| \bigg\{ 2^{\frac{k}{r}} \bigg( \int_{x_k}^{x_{k+1}}	\bigg( \int_{x_k}^t h \bigg)^q w(t)\,dt \bigg)^{\frac{1}{q}} \bigg\} \bigg\|_{\ell^r(\z)} \notag \\
	& \hspace{-2.5cm} + \bigg\| \bigg\{ 2^{\frac{k}{r}} \bigg( \int_{x_k}^{x_{k+1}} w \bigg)^{\frac{1}{q}} \bigg( \int_0^{x_k} h \bigg) \bigg\} \bigg\|_{\ell^r(\z)}. \label{eq.0.0.1}
	\end{align}
	
	Since
	\begin{align*}
	\bigg\| \bigg\{ 2^{k / r} \bigg( \int_{x_k}^{x_{k+1}} w \bigg)^{\frac{1}{q}} \bigg( \int_0^{x_k} h \bigg) \bigg\} \bigg\|_{\ell^r(\z)} & \\
	& \hspace{-3cm} \lesssim \bigg\| \bigg\{\bigg( \int_{x_{k-1}}^{x_k} u(x)\,dx \cdot \bigg[ \sup_{s \in [x_k,x_{k+1})} \bigg( \int_s^{x_{k+1}} w \bigg)^{\frac{1}{q}} \bigg( \int_0^s h \bigg) \bigg]^r \bigg)^{\frac{1}{r}} \bigg\} \bigg\|_{\ell^r(\z)} \\
	& \hspace{-3cm} \le \bigg\| \bigg\{ \bigg( \int_{x_{k-1}}^{x_k} u(x) \bigg[ \sup_{s \in [x,x_{k+1})} \bigg(
	\int_s^{x_{k+1}} w \bigg)^{\frac{1}{q}} \bigg( \int_0^s h \bigg) \bigg]^r \,dx\bigg)^{\frac{1}{r}} \bigg\} \bigg\|_{\ell^r(\z)} \\
	& \hspace{-3cm} \le \bigg\| \bigg\{ \bigg( \int_{x_{k-1}}^{x_k} u(x) \bigg[ \sup_{s \in [x,\infty)} \bigg(
	\int_s^{\infty} w \bigg)^{\frac{1}{q}} \bigg( \int_0^s h\bigg) \bigg]^r \,dx\bigg)^{\frac{1}{r}} \bigg\} \bigg\|_{\ell^r(\z)} \\
	& \hspace{-3cm} \le \bigg( \int_0^{\infty}  u(x) \bigg[ \sup_{s \in [x,\infty)} \bigg(
	\int_s^{\infty} w \bigg)^{\frac{1}{q}} \bigg( \int_0^s h\bigg) \bigg]^r \,dx \bigg)^{\frac{1}{r}},
	\end{align*}
	on using \eqref{eq.0.0.1}, we arrive at
	\begin{align*}
	\bigg( \int_0^{\infty} \bigg( \int_x^{\infty} \bigg( \int_0^t h \bigg)^q w(t)\,dt
	\bigg)^{\frac{r}{q}} u(x)\,dx \bigg)^{\frac{1}{r}} & \\
	& \hspace{-3cm} \lesssim \, \bigg\| \bigg\{ 2^{\frac{k}{r}} \bigg(
	\int_{x_k}^{x_{k+1}} \bigg( \int_{x_k}^t h \bigg)^q
	w(t)\,dt \bigg)^{\frac{1}{q}} \bigg\} \bigg\|_{\ell^r(\z)} \\
	& \hspace{-2.5cm} +  \bigg( \int_0^{\infty}  u(x) \bigg[ \sup_{s \in [x,\infty)} \bigg(
	\int_s^{\infty} w \bigg)^{\frac{1}{q}} \bigg( \int_0^s h\bigg) \bigg]^r \,dx \bigg)^{\frac{1}{r}}. 
	\end{align*}
	
	For the converse, note that, by \eqref{eq.0.0.1},
	\begin{align*}
	\bigg\| \bigg\{ 2^{\frac{k}{r}} \bigg(
	\int_{x_k}^{x_{k+1}} \bigg( \int_{x_k}^t h \bigg)^q
	w(t)\,dt\bigg)^{\frac{1}{q}} \bigg\} \bigg\|_{\ell^r(\z)} & \\
	& \hspace{-3cm} \lesssim \, \bigg( \int_0^{\infty} \bigg( \int_x^{\infty} \bigg( \int_0^t h \bigg)^q w(t)\,dt
	\bigg)^{\frac{r}{q}} u(x)\,dx \bigg)^{\frac{1}{r}}.
	\end{align*}
	
	On the other hand,
	\begin{align*}
	\bigg( \int_0^{\infty}  u(x) \bigg[ \sup_{s \in [x,\infty)} \bigg(
	\int_s^{\infty} w \bigg)^{\frac{1}{q}} \bigg( \int_0^s h\bigg) \bigg]^r \,dx\bigg)^{\frac{1}{r}} & \\
	& \hspace{-3cm} \le \bigg( \int_0^{\infty}  u(x) \bigg[ \sup_{s \in [x,\infty)} \bigg(
	\int_s^{\infty} \bigg( \int_0^t h \bigg)^q w(t)\,dt \bigg)^{\frac{1}{q}} \bigg]^r \,dx \bigg)^{\frac{1}{r}} \\
	& \hspace{-3cm} = \bigg( \int_0^{\infty}  u(x) \bigg(
	\int_x^{\infty} \bigg( \int_0^t h \bigg)^q w(t)\,dt \bigg)^{\frac{r}{q}}  \,dx \bigg)^{1/r}.
	\end{align*}
	
	We arrive at 
	\begin{align*}
	\bigg\| \bigg\{ 2^{\frac{k}{r}} \bigg(
	\int_{x_k}^{x_{k+1}} \bigg( \int_{x_k}^t h \bigg)^q
	w(t)\,dt \bigg)^{\frac{1}{q}} \bigg\} \bigg\|_{\ell^r(\z)} & \\
	& \hspace{-4.5cm} +  \bigg( \int_0^{\infty}  u(x) \bigg[ \sup_{s \in [x,\infty)} \bigg(
	\int_s^{\infty} w \bigg)^{\frac{1}{q}} \bigg( \int_0^s h\bigg) \bigg]^r \,dx \bigg)^{\frac{1}{r}} \\
	& \hspace{-5cm}  \lesssim  \bigg( \int_0^{\infty}  u(x) \bigg(
	\int_x^{\infty} \bigg( \int_0^t h \bigg)^q w(t)\,dt \bigg)^{\frac{r}{q}} \,dx \bigg)^{\frac{1}{r}}
	\end{align*}
	by combining the previous two inequalities.	
	
	The proof is completed. 
\end{proof}

So, we are able to formulate our reduction statement.
\begin{thm}\label{thm.1.5}
	Let $0 < q,\, r < \infty$ and  $u,\,v,\,w \in {\mathcal W}\I$.	Assume that $\{x_k\}_{k \in \z}$ is a covering sequence.
	Then inequality \eqref{main} holds if and only if both of inequalities \eqref{ineq.discr} and \eqref{ineq.supr} hold.
\end{thm}

\begin{proof}
The proof of the statement immediately follows from Theorem \ref{thm.1.4}.
\end{proof}	

		
\section{Solution of inequality \eqref{ineq.supr}}

In this section we present the solution of inequality\eqref{ineq.supr}.

\begin{thm}\label{thm44b+1}
	Let $1 \le p < \infty$, $0 < r < \infty$. Assume that $u,\,v,\,w \in \W (0,\infty)$.
	
	{\rm (a)} Let $p \le r$. Then inequality \eqref{ineq.supr} holds for all $ g\in \mathfrak{M}^{+} (0,\infty)$ if and only if
	\begin{equation*}
	D_1 := \sup_{x \in (0,\infty)} \bigg( \int_0^x u \bigg)^{\frac{1}{r}} \bigg( \int_x^{\infty} w \bigg)^{\frac{1}{q}} \sigma_p(0,x) < \infty
	\end{equation*}
	and
	\begin{equation*}
	D_2 := \sup_{x \in (0,\infty)} \bigg( \int_x^{\infty} \bigg( \int_t^{\infty} w \bigg)^{\frac{r}{q}} u(t)\,dt \bigg)^{\frac{1}{r}} \sigma_p(0,x) < \infty.
	\end{equation*}
	Moreover, the least constant $C$ such that \eqref{ineq.supr} holds for all $ g\in \mathfrak{M}^{+} $ satisfies $ C \approx D_1 + D_2$.
	
	{\rm (b)} Let $r < p$. Then inequality \eqref{ineq.supr} holds for all $ h \in \mathfrak{M}^{+} (0,\infty)$ if and only if
	\begin{equation*}
	D_3 := \bigg( \int_0^{\infty} \bigg[\sup_{\tau \in [x,\infty)} \bigg( \int_{\tau}^{\infty} w \bigg)^{\frac{1}{q}} \sigma_p(0,\tau) \bigg]^{\frac{pr}{p-r}} \bigg( \int_0^x u \bigg)^{\frac{r}{p-r}} u(x) \,dx \bigg)^{\frac{p-r}{pr}} < \infty
	\end{equation*}
	and
	\begin{equation*}
	D_4 := \bigg(\int_0^{\infty} \bigg(\int_{x}^{\infty} \bigg( \int_y^{\infty} w \bigg)^{\frac{r}{q}} u(y) \, dy\bigg)^{\frac{r}{p-r}} \bigg( \int_x^{\infty} w	\bigg)^{\frac{r}{q}} [\sigma_p(0,x)]^{\frac{pr}{p-r}} \, u(x)\,dx \bigg)^{\frac{pr}{p-r}} < \infty.
	\end{equation*}
	Moreover, the least constant $C$ such that \eqref{ineq.supr} holds for all $ g\in \mathfrak{M}^{+} $ satisfies $ C \approx D_3 + D_4$.
\end{thm}

\begin{proof}
The statement directly follows from Theorem \ref{thm44b}.	
\end{proof}
	

\section{Discrete solution of inequality \eqref{ineq.discr}.}\label{s.5}
		
Now we give a discrete characterization of inequality \eqref{ineq.discr}.
\begin{lem}\label{lem.1.6}
	Let $1 \le p < \infty$, $0 < q,\, r < \infty$ and  $u,\,v,\,w \in {\mathcal W}(0,\infty)$. Assume that $\{x_k\}_{k \in \z}$ is a covering sequence.
	
	{\rm (i)} If $p \le q$, then inequality \eqref{ineq.discr} holds with constant independent of $h \in {\mathfrak M}^+(0,\infty)$ if and only if $A_1 < \infty$, where
	$$
	A_1 := \bigg\| \bigg\{ 2^{\frac{k}{r}} \bigg( \esup_{x \in [x_k,x_{k+1})} \bigg( \int_x^{x_{k+1}} w \bigg)^{\frac{1}{q}} \sigma_p(x_k,x) \bigg) \bigg\} \bigg\|_{\ell^{\rho}(\z)}.
	$$
	Moreover, if $C$ is the best constant in \eqref{ineq.discr}, then $C \approx A_1$.	
	
	{\rm (ii)} If $q < p$, then inequality \eqref{ineq.discr} holds with constant independent of $h \in {\mathfrak M}^+(0,\infty)$ if and only if $A_2 < \infty$, where
	$$
	A_2 := \bigg\| \bigg\{ 2^{\frac{k}{r}} \bigg( \int_{x_k}^{x_{k+1}} \bigg( \int_x^{x_{k+1}} w \bigg)^{\frac{q}{p-q}} w(x) \big[ \sigma_p(x_k,x) \big]^{\frac{pq}{p-q}} \,dx \bigg)^{\frac{p-q}{pq}} \bigg\} \bigg\|_{\ell^{\rho}(\z)}.
	$$
	Moreover, if $C$ is the best constant in \eqref{ineq.discr}, then $C \approx A_2$.	
\end{lem}

\begin{proof}
	We give the proof of the second case. The first one can be proved similarly, so its proof is omitted. 
	
	{\bf Necessity:}  Suppose that inequality \eqref{ineq.discr}	holds with constant $C$ independent of $h \in {\mathfrak M}^+(0,\infty)$. 
	
	Since, by Theorem \ref{thm.Copson}, (ii), for any $k \in \Z$
	\begin{align*}
	\sup \bigg\{ \bigg( \int_{x_k}^{x_{k+1}} \bigg( \int_{x_k}^x h \bigg)^q w(x)\,dx\bigg)^{\frac{1}{q}} :\,  \int_{x_k}^{x_{k+1}} h^p v = 1 \bigg\} & \\
	& \hspace{-5cm} = \bigg( \int_{x_k}^{x_{k+1}} \bigg( \int_x^{x_{k+1}} w \bigg)^{\frac{q}{p-q}} w(x) \big[ \sigma_p(x_k,x) \big]^{\frac{pq}{p-q}} \,dx \bigg)^{\frac{p-q}{pq}},
	\end{align*}
	then there exists $h_k \in \mp^+(0,\infty)$ such that $\supp h_k \in (x_k,x_{k+1})$, $\int_{x_k}^{x_{k+1}} h_k^p v = 1$ and
	$$
	\bigg(	\int_{x_k}^{x_{k+1}} \bigg( \int_{x_k}^t h_k \bigg)^q
	w(t)\,dt\bigg)^{\frac{1}{q}} \ge \frac{1}{2} \bigg( \int_{x_k}^{x_{k+1}} \bigg( \int_x^{x_{k+1}} w \bigg)^{\frac{q}{p-q}} w(x) \big[ \sigma_p(x_k,x) \big]^{\frac{pq}{p-q}} \,dx \bigg)^{\frac{p-q}{pq}}.
	$$
	
	Define
	\begin{equation}\label{g}
	h = \sum_{m \in \Z} a_m h_m,
	\end{equation}
	where $\{a_k\}_{k \in \Z}$ is any sequence of nonnegative numbers.

	Since $h \in {\mathfrak M}^+(0,\infty)$, then inequality \eqref{ineq.discr} holds for function \eqref{g} as well. Thus the inequality
	\begin{equation*}
	\bigg\| \bigg\{ a_k 2^{\frac{k}{r}} \bigg( \int_{x_k}^{x_{k+1}} \bigg( \int_x^{x_{k+1}} w \bigg)^{\frac{q}{p-q}} w(x) \big[ \sigma_p(x_k,x) \big]^{\frac{pq}{p-q}} \,dx \bigg)^{\frac{p-q}{pq}} \bigg\} \bigg\|_{\ell^r (\z)} \le C \, \|\{a_k\}\|_{\ell^p(\z)}
	\end{equation*}
	holds for all sequences $\{a_k\}_{k \in \Z}$ of nonnegative numbers. 
	
	By Proposition \ref{prop.2.1}, we arrive at
	$$
	A_2 = \bigg\| \bigg\{ 2^{\frac{k}{r}} \bigg( \int_{x_k}^{x_{k+1}} \bigg( \int_x^{x_{k+1}} w \bigg)^{\frac{q}{p-q}} w(x) \big[ \sigma_p(x_k,x) \big]^{\frac{pq}{p-q}} \,dx \bigg)^{\frac{p-q}{pq}} \bigg\} \bigg\|_{\ell^{\rho} (\z)}
	\le C.
	$$
	
	{\bf Sufficiency:} Assume that $A_2 < \infty$.	By Theorem \ref{thm.Copson}, (ii), applying inequality \eqref{eq.1.1}, we have that
	\begin{align*}
	\bigg\| \bigg\{ 2^{\frac{k}{r}} \bigg( \int_{x_k}^{x_{k+1}} \bigg( \int_s^{x_{k+1}} h \bigg)^q
	w(s)\,ds\bigg)^{\frac{1}{q}} \bigg\} \bigg\|_{\ell^r(\z)} & \\
	& \hspace{-5cm} \le \bigg\| \bigg\{ 2^{\frac{k}{r}} \bigg( \int_{x_k}^{x_{k+1}} \bigg( \int_x^{x_{k+1}} w \bigg)^{\frac{q}{p-q}} w(x) \big[ \sigma_p(x_k,x) \big]^{\frac{pq}{p-q}} \,dx \bigg)^{\frac{p-q}{pq}} \bigg( \int_{x_k}^{x_{k+1}} h^p v\bigg)^{\frac{1}{p}} \bigg\} \bigg\|_{\ell^r(\z)} \\
	& \hspace{-5cm} = A_2 \, \bigg\| \bigg\{ \bigg( \int_{x_n}^{x_{n+1}} h^p v \bigg)^{\frac{1}{p}} \bigg\} \bigg\|_{\ell^p (\z)}.
	\end{align*} 
	
	Thus, inequality \eqref{ineq.discr} holds, and, if $C$ is the best constant in \eqref{ineq.discr}, then 		
	$$
	C \le A_2.
	$$	
	
	The proof is completed.
\end{proof}


\section{Continuous sufficient conditions for inequality \eqref{ineq.discr}}\label{s.6}

\begin{lem}\label{lem.1.8}
	Let $1 \le p < \infty$, $0 < q,\, r < \infty$ and  $u,\,v,\,w \in {\mathcal W}(0,\infty)$. 
	
	{\rm (i)} Let $p \le \min\{q,\,r\}$. If
	$$
	B_1 : = \sup_{t \in (0,\infty)} \bigg( \int_0^t u \bigg)^{\frac{1}{r}} \bigg( \int_t^{\infty} w \bigg)^{\frac{1}{q}} \sigma_p(0,t) < \infty,
	$$
	then inequality \eqref{ineq.discr}	holds with constant independent of $h \in {\mathfrak M}^+(0,\infty)$. Moreover, if $C$ is the best constant in \eqref{ineq.discr}, then $C \lesssim B_1$.	
	
	{\rm (ii)} Let $r < p \le q$. If
	$$
	B_2 : = \bigg( \int_0^{\infty} \bigg( \int_0^t u \bigg)^{\frac{r}{p-r}} u(t) \bigg( \esup_{x \in (t,\infty)} \bigg( \int_x^{\infty} w \bigg)^{\frac{1}{q}} \sigma_p(0,x) \bigg)^{\frac{pr}{p-r}} \,dt \bigg)^{\frac{p-r}{pr}} < \infty,
	$$
	then inequality \eqref{ineq.discr}	holds with constant independent of $h \in {\mathfrak M}^+(0,\infty)$. Moreover, if $C$ is the best constant in \eqref{ineq.discr}, then $C \lesssim B_2$.	
	
	{\rm (iii)} Let $q < p \le r$. If
	$$
	B_3 : = \sup_{t \in (0,\infty)} \bigg( \int_0^t u \bigg)^{\frac{1}{r}} \bigg( \int_t^{\infty} \bigg( \int_x^{\infty} w \bigg)^{\frac{q}{p-q}} w(x) \big[ \sigma_p(0,x) \big]^{\frac{pq}{p-q}} \,dx\bigg)^{\frac{p-q}{pq}} < \infty,
	$$
	then inequality \eqref{ineq.discr}	holds with constant independent of $h \in {\mathfrak M}^+(0,\infty)$. Moreover, if $C$ is the best constant in \eqref{ineq.discr}, then $C \lesssim B_3$.	
	
	{\rm (iv)} Let $\max\{q,\,r\} < p$. If
	$$
	B_4 : = \bigg( \int_0^{\infty} \bigg( \int_0^t u \bigg)^{\frac{r}{p-r}} u(t) \bigg( \int_t^{\infty} \bigg( \int_x^{\infty} w \bigg)^{\frac{q}{p-q}} w(x) \big[ \sigma_p(0,x) \big]^{\frac{pq}{p-q}} \,dx\bigg)^{\frac{r(p-q)}{q(p-r)}}\,dt \bigg)^{\frac{p-r}{pr}} < \infty,
	$$
	then inequality \eqref{ineq.discr}	holds with constant independent of $h \in {\mathfrak M}^+(0,\infty)$. Moreover, if $C$ is the best constant in \eqref{ineq.discr}, then $C \lesssim B_4$.	
\end{lem}

\begin{proof}
	{\rm (i)} Let $p \le \min\{q,\,r\}$. Assume that $B_1 < \infty$. 

	Recall that, if $F$ is a non-negative non-decreasing function on $(0,\infty)$, then
	\begin{equation}\label{Fubini.2}	
	\esup_{t \in (0,\infty)} F(t)G(t) = \esup_{t \in (0,\infty)} F(t)
	\esup_{\tau \in (t,\infty)} G(\tau)
	\end{equation}
	(see, for instance, \cite{gp2}).
		
	On using \eqref{Fubini.2}, we get that
	\begin{align*}
	A_1 & = \sup_{k \in \z} 2^{\frac{k}{r}} \bigg( \esup_{x \in (x_k,x_{k+1})} \bigg( \int_x^{x_{k+1}} w \bigg)^{\frac{1}{q}} \sigma_p(x_k,x) \bigg) \\
	& \lesssim \sup_{k \in \z} \bigg( \int_0^{x_k} u \bigg)^{\frac{1}{r}} \bigg( \esup_{x \in (x_k,\infty)} \bigg( \int_x^{\infty} w \bigg)^{\frac{1}{q}} \sigma_p(0,x) \bigg) \\
	& \le \sup_{t \in (0,\infty)} \bigg( \int_0^t u \bigg)^{\frac{1}{r}} \bigg( \esup_{x \in (t,\infty)} \bigg( \int_x^{\infty} w \bigg)^{\frac{1}{q}} \sigma_p(0,x) \bigg) \\
	& = \sup_{t \in (0,\infty)} \bigg( \int_0^t u \bigg)^{\frac{1}{r}} \bigg( \int_t^{\infty} w \bigg)^{\frac{1}{q}} \sigma_p(0,t) = B_1.
	\end{align*}
	Thus $A_1 < \infty$, and the statement follows from Lemma \ref{lem.1.6}, (i).
	
	{\rm (ii)} Let $r < p \le q$. Assume that $B_2 < \infty$. Clearly,
	\begin{align*}
	A_1 & = \bigg( \sum_{k \in \z} 2^{k\frac{p}{p-r}} \bigg( \esup_{x \in (x_k,x_{k+1})} \bigg( \int_x^{x_{k+1}} w \bigg)^{\frac{1}{q}} \sigma_p(x_k,x) \bigg)^{\frac{pr}{p-r}} \bigg)^{\frac{p-r}{pr}} \\
	& \lesssim \bigg( \sum_{k \in \z} \int_{x_{k-1}}^{x_k} \bigg( \int_0^t u \bigg)^{\frac{r}{p-r}} u(t)\,dt \cdot \bigg( \esup_{x \in (x_k,\infty)} \bigg( \int_x^{\infty} w \bigg)^{\frac{1}{q}} \sigma_p(0,x) \bigg)^{\frac{pr}{p-r}} \bigg)^{\frac{p-r}{pr}} \\
	& \le \bigg( \sum_{k \in \z} \int_{x_{k-1}}^{x_k} \bigg( \int_0^t u \bigg)^{\frac{r}{p-r}} u(t) \bigg( \esup_{x \in (t,\infty)} \bigg( \int_x^{\infty} w \bigg)^{\frac{1}{q}} \sigma_p(0,x) \bigg)^{\frac{pr}{p-r}} \,dt \bigg)^{\frac{p-r}{pr}} \\
	& \le \bigg( \int_0^{\infty} \bigg( \int_0^t u \bigg)^{\frac{r}{p-r}} u(t) \bigg( \esup_{x \in (t,\infty)} \bigg( \int_x^{\infty} w \bigg)^{\frac{1}{q}} \sigma_p(0,x) \bigg)^{\frac{pr}{p-r}} \,dt \bigg)^{\frac{p-r}{pr}} = B_2.
	\end{align*}
	Thus $A_1 < \infty$, and the statement follows from Lemma \ref{lem.1.6}, (i).
	
	{\rm (iii)} Let $q < p \le r$. Assume that $B_3 < \infty$. We have that
    \begin{align*}
	A_2 & = \sup_{k \in \z} 2^{\frac{k}{r}} \bigg( \int_{x_k}^{x_{k+1}} \bigg( \int_x^{x_{k+1}} w \bigg)^{\frac{q}{p-q}} w(x) \big[ \sigma_p(x_k,x) \big]^{\frac{pq}{p-q}} \,dx\bigg)^{\frac{p-q}{pq}}  \\
	& \lesssim \sup_{k \in \z} \bigg( \int_0^{x_k} u \bigg)^{\frac{1}{r}} \bigg( \int_{x_k}^{\infty} \bigg( \int_x^{\infty} w \bigg)^{\frac{q}{p-q}} w(x) \big[ \sigma_p(0,x) \big]^{\frac{pq}{p-q}} \,dx\bigg)^{\frac{p-q}{pq}}  \\
	& \le \sup_{t \in (0,\infty)} \bigg( \int_0^t u \bigg)^{\frac{1}{r}} \bigg( \int_t^{\infty} \bigg( \int_x^{\infty} w \bigg)^{\frac{q}{p-q}} w(x) \big[ \sigma_p(0,x) \big]^{\frac{pq}{p-q}} \,dx\bigg)^{\frac{p-q}{pq}} = B_3.
	\end{align*}
	Thus $A_2 < \infty$, and the statement follows from Lemma \ref{lem.1.6}, (ii).	
	
	{\rm (iv)} Let $\max\{q,\,r\} < p$. Assume that $B_4 < \infty$. We have that
	\begin{align*}
	A_2 & = \bigg( \sum_{k \in \z} 2^{k\frac{p}{p-r}} \bigg( \int_{x_k}^{x_{k+1}} \bigg( \int_x^{x_{k+1}} w \bigg)^{\frac{q}{p-q}} w(x) \big[ \sigma_p(x_k,x) \big]^{\frac{pq}{p-q}} \,dx\bigg)^{\frac{r(p-q)}{q(p-r)}} \bigg)^{\frac{p-r}{pr}} \\
	& \lesssim \bigg( \sum_{k \in \z} \int_{x_{k-1}}^{x_k} \bigg( \int_0^t u \bigg)^{\frac{r}{p-r}} u(t)\,dt  \cdot \bigg( \int_{x_k}^{\infty} \bigg( \int_x^{\infty} w \bigg)^{\frac{q}{p-q}} w(x) \big[ \sigma_p(0,x) \big]^{\frac{pq}{p-q}} \,dx \bigg)^{\frac{r(p-q)}{q(p-r)}} \bigg)^{\frac{p-r}{pr}} \\
	& \le \bigg( \sum_{k \in \z} \int_{x_{k-1}}^{x_k} \bigg( \int_0^t u \bigg)^{\frac{r}{p-r}} u(t) \bigg( \int_t^{\infty} \bigg( \int_x^{\infty} w \bigg)^{\frac{q}{p-q}} w(x) \big[ \sigma_p(0,x) \big]^{\frac{pq}{p-q}} \,dx\bigg)^{\frac{r(p-q)}{q(p-r)}}\,dt \bigg)^{\frac{p-r}{pr}} \\
	& \le \bigg( \int_0^{\infty} \bigg( \int_0^t u \bigg)^{\frac{r}{p-r}} u(t) \bigg( \int_t^{\infty} \bigg( \int_x^{\infty} w \bigg)^{\frac{q}{p-q}} w(x) \big[ \sigma_p(0,x) \big]^{\frac{pq}{p-q}} \,dx\bigg)^{\frac{r(p-q)}{q(p-r)}}\,dt \bigg)^{\frac{p-r}{pr}} = B_4. 
	\end{align*}
	Thus $A_2 < \infty$, and the statement follows from Lemma \ref{lem.1.6}, (ii).
	
	The proof is completed.
\end{proof}	


\section{Necessity of conditions $B_i$, $i=\overline{1,4}$ for inequality \eqref{main}}\label{s.7}

In this section, we show that conditions obtained in previous section are necessary for inequality \eqref{main}.		
\begin{lem}\label{lem.1.9}
	Let $1 \le p < \infty$, $0 < q,\, r < \infty$ and  $u,\,v,\,w \in {\mathcal W}(0,\infty)$. Assume that inequality \eqref{main} holds.
	
	{\rm (i)} If $p \le \min\{q,\,r\}$, then $B_1 < \infty$. Moreover, if $C$ is the best constant in \eqref{main}, then $B_1 \lesssim C$.
	
	{\rm (ii)} If $r < p \le q$, then $B_2 < \infty$. Moreover, if $C$ is the best constant in \eqref{main}, then $B_2 \lesssim C$.
	
	{\rm (iii)} If $q < p \le r$, then $B_3 < \infty$. Moreover, if $C$ is the best constant in \eqref{main}, then $B_3 \lesssim C$.
	
	{\rm (iv)} If $\max\{q,\,r\} < p$, then $B_4 < \infty$. Moreover, if $C$ is the best constant in \eqref{main}, then $B_4 \lesssim C$.
\end{lem}

\begin{proof}
Let $1 \le p < \infty$, $0 < q,\, r < \infty$. Assume that inequality \eqref{main} holds. Suppose that $\{x_k\}_{k \in \z}$ is a covering sequence.

{\rm (i)} Let $p \le \min\{q,\,r\}$.  Since $B_1 = D_1$, then the statement follows by Theorem \ref{thm.1.5} and Theorem \ref{thm44b+1}.

{\rm (ii)} Let $r < p \le q$. We get, by applying Lemma \ref{L:1.2}, that
\begin{align*}
B_2 \lesssim \bigg( \sum_{k \in \z} 2^{k\frac{p}{p-r}} \bigg[ \esup_{x \in [x_k,x_{k+1})} \bigg( \int_x^{\infty} w \bigg)^{\frac{1}{q}} \sigma_p(0,x) \bigg]^{\frac{pr}{p-r}} \bigg)^{\frac{p-r}{pr}}. 
\end{align*}

It is easy to see that for any $k \in \zstroke$ 
\begin{align*}
\esup_{x \in [x_k,x_{k+1})} \bigg( \int_x^{\infty} w \bigg)^{\frac{1}{q}} \sigma_p(0,x) \lesssim \esup_{x \in [x_k,x_{k+1})} \bigg( \int_x^{x_{k+1}} w \bigg)^{\frac{1}{q}} \sigma_p(x_k,x) + \sup_{\tau \in [x_k,\infty)} \bigg( \int_{\tau}^{\infty} w \bigg)^{\frac{1}{q}} \sigma_p(0,\tau) 
\end{align*}

Thus
\begin{align*}
B_2 \lesssim & \bigg( \sum_{k \in \z} 2^{k\frac{p}{p-r}} \bigg[ \esup_{x \in [x_k,x_{k+1})} \bigg( \int_x^{x_{k+1}} w \bigg)^{\frac{1}{q}} \sigma_p(x_k,x) \bigg]^{\frac{pr}{p-r}} \bigg)^{\frac{p-r}{pr}} \\
& + \bigg( \sum_{k \in \zstroke} 2^{k\frac{p}{p-r}} \bigg[ \sup_{\tau \in [x_k,\infty)} \bigg( \int_{\tau}^{\infty} w \bigg)^{\frac{1}{q}} \sigma_p(0,\tau) \bigg]^{\frac{pr}{p-r}} \bigg)^{\frac{p-r}{pr}} \\
= & A_1 + \bigg( \sum_{k \in \zstroke} 2^{k\frac{p}{p-r}} \bigg[ \sup_{\tau \in [x_k,\infty)} \bigg( \int_{\tau}^{\infty} w \bigg)^{\frac{1}{q}} \sigma_p(0,\tau) \bigg]^{\frac{pr}{p-r}} \bigg)^{\frac{p-r}{pr}}.
\end{align*}

Denote by
$$
E_1 : = \bigg( \sum_{k \in \zstroke} 2^{k\frac{p}{p-r}} \bigg[ \sup_{\tau \in [x_k,\infty)} \bigg( \int_{\tau}^{\infty} w \bigg)^{\frac{1}{q}} \sigma_p(0,\tau) \bigg]^{\frac{pr}{p-r}} \bigg)^{\frac{p-r}{pr}}.
$$

Clearly, 
\begin{align*}
E_1 & \approx \bigg( \sum_{k \in \zstroke} \bigg( \int_{x_{k-1}}^{x_k} \bigg( \int_{x_{k-1}}^t u \bigg)^{\frac{r}{p-r}} u(t)\,dt \bigg) \bigg[ \sup_{\tau \in [x_k,\infty)} \bigg( \int_{\tau}^{\infty} w \bigg)^{\frac{1}{q}} \sigma_p(0,\tau) \bigg]^{\frac{pr}{p-r}} \bigg)^{\frac{p-r}{pr}} \\
& \le \bigg( \sum_{k \in \z} \int_{x_{k-1}}^{x_k} \bigg[\sup_{\tau \in [x,\infty)} \bigg( \int_{\tau}^{\infty} w \bigg)^{\frac{1}{q}} \sigma_p(0,\tau) \bigg]^{\frac{pr}{p-r}} \bigg( \int_0^x u \bigg)^{\frac{r}{p-r}} u(x) \,dx \bigg)^{\frac{p-r}{pr}} \\
& \le \bigg( \int_0^{\infty} \bigg[\sup_{\tau \in [x,\infty)} \bigg( \int_{\tau}^{\infty} w \bigg)^{\frac{1}{q}} \sigma_p(0,\tau) \bigg]^{\frac{pr}{p-r}} \bigg( \int_0^x u \bigg)^{\frac{r}{p-r}} u(x) \,dx \bigg)^{\frac{p-r}{pr}} = D_3.
\end{align*}

Combining yields
$$
B_2 \lesssim A_1 + D_3,
$$
and the statement follows by Theorem \ref{thm.1.5}, Lemma \ref{lem.1.6} and Theorem \ref{thm44b+1}.

{\rm (iii)} Let $q < p \le r$. Applying Lemma \ref{L:1.2}, we get that
\begin{align*}
B_3 \lesssim \sup_{k \in \z} 2^{\frac{k}{r}} \bigg( \int_{x_k}^{x_{k+1}} \bigg( \int_x^{\infty} w \bigg)^{\frac{q}{p-q}} w(x) \big[ \sigma_p(0,x) \big]^{\frac{pq}{p-q}} \,dx\bigg)^{\frac{p-q}{pq}}.
\end{align*}

Integrating by parts, for any $k \in \zstroke$, it is easy to see that
\begin{align}
\bigg( \int_{x_k}^{x_{k+1}} \bigg( \int_x^{\infty} w \bigg)^{\frac{q}{p-q}} w(x) \big[ \sigma_p(0,x) \big]^{\frac{pq}{p-q}} \,dx\bigg)^{\frac{p-q}{pq}} & \notag \\
& \hspace{-5cm} \lesssim \bigg( \int_{x_k}^{x_{k+1}} \bigg( \int_x^{x_{k+1}} w \bigg)^{\frac{q}{p-q}} w(x) \big[ \sigma_p(x_k,x) \big]^{\frac{pq}{p-q}} \,dx \bigg)^{\frac{p-q}{pq}} + \sup_{t \in [x_k,\infty)} \bigg( \int_t^{\infty} w \bigg)^{\frac{1}{q}} \sigma_p(0,t). \label{eq.1.1.1.1}
\end{align}

Thus,
\begin{align*}
B_3 \lesssim & \sup_{k \in \z} 2^{\frac{k}{r}} \bigg( \int_{x_k}^{x_{k+1}} \bigg( \int_x^{x_{k+1}} w \bigg)^{\frac{q}{p-q}} w(x) \big[ \sigma_p(x_k,x) \big]^{\frac{pq}{p-q}} \,dx \bigg)^{\frac{p-q}{pq}} \\
& + \sup_{k \in \z} 2^{\frac{k}{r}} \sup_{t \in [x_k,\infty)} \bigg( \int_t^{\infty} w \bigg)^{\frac{1}{q}} \sigma_p(0,t) \\
= & A_2 + \sup_{k \in \z} 2^{\frac{k}{r}} \sup_{t \in [x_k,\infty)} \bigg( \int_t^{\infty} w \bigg)^{\frac{1}{q}} \sigma_p(0,t).
\end{align*}

Denote by
$$
E_2 : = \sup_{k \in \z} 2^{\frac{k}{r}} \sup_{t \in [x_k,\infty)} \bigg( \int_t^{\infty} w \bigg)^{\frac{1}{q}} \sigma_p(0,t).
$$

Since
\begin{align*}
E_2 & = \sup_{k \in \z} \bigg( \int_0^{x_k} u \bigg)^{\frac{1}{r}} \sup_{t \in [x_k,\infty)} \bigg( \int_t^{\infty} w \bigg)^{\frac{1}{q}} \sigma_p(0,t) \\
& \le \sup_{x \in (0,\infty)} \bigg( \int_0^x u \bigg)^{\frac{1}{r}} \sup_{t \in [x,\infty)} \bigg( \int_t^{\infty} w \bigg)^{\frac{1}{q}} \sigma_p(0,t) = D_1,
\end{align*}
we arrive at
$$
B_3 \lesssim A_2 + D_1
$$
and the statement follows by Theorem \ref{thm.1.5}, Lemma \ref{lem.1.6} and Theorem \ref{thm44b+1}.

{\rm (iv)} Let $\max\{q,\,r\} < p$. Applying Lemma \ref{L:1.2}, we get that
\begin{align*}
B_4 \lesssim \bigg( \sum_{k \in \z} 2^{k\frac{p}{p-r}} \bigg( \int_{x_k}^{x_{k+1}} \bigg( \int_x^{\infty} w \bigg)^{\frac{q}{p-q}} w(x) \big[ \sigma_p(0,x) \big]^{\frac{pq}{p-q}} \,dx \bigg)^{\frac{r(p-q)}{q(p-r)}} \bigg)^{\frac{p-r}{pr}}.
\end{align*}

On using \eqref{eq.1.1.1.1}, we have that
\begin{align*}
B_4 \lesssim & \bigg( \sum_{k \in \z} 2^{k\frac{p}{p-r}} \bigg( \int_{x_k}^{x_{k+1}} \bigg( \int_x^{x_{k+1}} w \bigg)^{\frac{q}{p-q}} w(x) \big[ \sigma_p(x_k,x) \big]^{\frac{pq}{p-q}} \,dx \bigg)^{\frac{r(p-q)}{q(p-r)}} \bigg)^{\frac{p-r}{pr}} \\
& + \bigg( \sum_{k \in \z} 2^{k\frac{p}{p-r}} \bigg[ \sup_{t \in [x_k,\infty)} \bigg( \int_t^{\infty} w \bigg)^{\frac{1}{q}} \sigma_p(0,t) \bigg]^{\frac{pr}{p-r}}\bigg)^{\frac{p-r}{pr}} =  A_2 + E_1 \lesssim A_2 + D_3, 
\end{align*}
and the statement follows by Theorem \ref{thm.1.5}, Lemma \ref{lem.1.6} and Theorem \ref{thm44b+1}.

The proof is completed.

\end{proof}


\section{Proof of the main statement.}\label{s.8}

We are now in position to prove our main result.

\noindent{\bf Proof of Theorem \ref{main1}.} 

Note that $F_1 = D_1$, $F_2 = D_2$, $F_3 = D_3$ and $F_4 = D_4$.

Obviously, for any $t \in (0,\infty)$
\begin{align}
\bigg( \int_{t}^{\infty} w \bigg)^{\frac{1}{q}} \sigma_p(0,t) & \approx \bigg( \int_t^{\infty} \bigg( \int_y^{\infty} w \bigg)^{\frac{q}{p-q}} w(y)\,dy \bigg)^{\frac{p-q}{pq}} \sigma_p(0,t) \notag \\
& \le \bigg( \int_t^{\infty} \bigg( \int_y^{\infty} w \bigg)^{\frac{q}{p-q}} w(y) \big[\sigma_p(0,y)\big]^{\frac{pq}{p-q}} \,dy \bigg)^{\frac{p-q}{pq}}. \label{eq.cond.control}
\end{align}

Thus
\begin{align*}
F_1 = B_1 & = \sup_{x \in (0,\infty)} \bigg( \int_0^x u \bigg)^{\frac{1}{r}} \bigg( \int_x^{\infty} w \bigg)^{\frac{1}{q}} \sigma_p(0,x) \\
& \le \sup_{x \in (0,\infty)} \bigg( \int_0^x u \bigg)^{\frac{1}{r}} \bigg( \int_x^{\infty} \bigg( \int_y^{\infty} w \bigg)^{\frac{q}{p-q}} w(y) \big[\sigma_p(0,y)\big]^{\frac{pq}{p-q}} \,dy \bigg)^{\frac{p-q}{pq}} = B_3 = F_5.
\end{align*}

By inequality \eqref{eq.cond.control}, we get that
\begin{align*}
\sup_{t \in [x,\infty)} \bigg( \int_t^{\infty} w \bigg)^{\frac{1}{q}} \sigma_p(0,t) & \le \sup_{t \in [x,\infty)} \bigg( \int_t^{\infty} \bigg( \int_y^{\infty} w \bigg)^{\frac{q}{p-q}} w(y) \big[\sigma_p(0,y)\big]^{\frac{pq}{p-q}} \,dy \bigg)^{\frac{p-q}{pq}} \\
& = \bigg( \int_x^{\infty} \bigg( \int_y^{\infty} w \bigg)^{\frac{q}{p-q}} w(y) \big[\sigma_p(0,y)\big]^{\frac{pq}{p-q}} \,dy \bigg)^{\frac{p-q}{pq}}.
\end{align*}

Thus
\begin{align*}
F_3 = B_2 & = \bigg( \int_0^{\infty} \bigg[\sup_{t \in [x,\infty)} \bigg( \int_t^{\infty} w \bigg)^{\frac{1}{q}} \sigma_p(0,t) \bigg]^{\frac{pr}{p-r}} \bigg( \int_0^x u \bigg)^{\frac{r}{p-r}} u(x) \,dx \bigg)^{\frac{p-r}{pr}} \\
& \le \bigg( \int_0^{\infty} \bigg( \int_0^x u \bigg)^{\frac{r}{p-r}} u(x) \bigg( \int_x^{\infty} \bigg( \int_y^{\infty} w \bigg)^{\frac{q}{p-q}} w(y) \big[ \sigma_p(0,y) \big]^{\frac{pq}{p-q}} \,dy\bigg)^{\frac{r(p-q)}{q(p-r)}}\,dx \bigg)^{\frac{p-r}{pr}} = B_4 = F_6.
\end{align*}

So, the proof of the statement immediately follows from Theorem \ref{thm.1.5} and \ref{thm44b+1}, Lemmas \ref{lem.1.8} and \ref{lem.1.9}. \qed



\end{document}